\documentclass[a4paper,leqno,11pt]{amsart}

\usepackage{amsfonts,amssymb,stackrel,verbatim,amsmath,amsthm,latexsym,textcomp,amscd}
\usepackage{latexsym,amsfonts,amssymb,epsfig,verbatim}
\usepackage{amsmath,amsthm,amssymb,latexsym,graphics,textcomp}
\usepackage{mathtools}
\usepackage{paralist}
\usepackage{graphicx}
\usepackage{color}
\usepackage{url}
\usepackage{enumerate}
\usepackage[mathscr]{euscript}
\usepackage{tikz-cd}
\usetikzlibrary{shapes.geometric}
\usepackage{dsfont}
\usetikzlibrary{matrix}
\usepackage{hyperref}

\input xy
\xyoption{all}

\usetikzlibrary{shapes,arrows,shapes.multipart}

\usepackage{amssymb}
\usepackage{latexsym}
\usepackage{enumerate}
\usepackage{amscd}
\usepackage{tikz}
\usepackage[all]{xy}
\usepackage{graphics}
\usepackage[active]{srcltx}
\usepackage{mathtools}

\renewcommand{\ker}{\text{Ker}}

\DeclareMathOperator{\tr}{tr}

\newcommand{\Z}{\mathbb{Z}}

\newcommand{\cF}{\mathcal{F}}
\newcommand{\cP}{\mathcal{P}}
\newcommand{\cFh}{\cF \langle H \rangle}
\newcommand{\cFk}{\cF \langle K \rangle}

\newcommand{\K}{\ensuremath{\mathcal{K}_4}}
\newcommand{\RO}{\ensuremath{\mathrm{RO}}}
\newcommand{\uFz}{\underline{\mathbb{F}_2}}
\newcommand{\uM}{\underline{M}}
\newcommand{\Fz}{\mathbb{F}_2}

\newcommand{\Burn}{\mathrm{Burn}}

\newcommand{\Sp}{\mathcal{S}p}

\newcommand{\vsn}{\vspace{2 mm}}

\newcommand{\op}{\mathrm{op}}

\newcommand{\res}{\mathrm{res}}

\newcommand{\uZ}{\underline{\mathbb{Z}}}

    \newcommand{\EPplus}{\ensuremath{E{\cP}_+} }
    \newcommand{\EPtilde}{\ensuremath{\widetilde{E{\cP}}} }
 
    \newcommand{\EKplus}{\ensuremath{E{\K}_+} }
    \newcommand{\EKtilde}{\ensuremath{\widetilde{E{\K}}} }
    
\newcommand{\hstar}[2][\bigstar]{\widetilde{H}_{\K}^{#1}(#2)}

\theoremstyle{plain}
\swapnumbers
    \newtheorem{theorem}[figure]{Theorem}
    \newtheorem{proposition}[figure]{Proposition}
    \newtheorem{lemma}[figure]{Lemma}
     \newtheorem{subsec}[figure]{}
    \newtheorem{corollary}[figure]{Corollary}

\theoremstyle{definition}
    \newtheorem{definition}[figure]{Definition}

\theoremstyle{remark}
        \newtheorem{remark}[figure]{Remark}
        
        \newtheorem{example}[figure]{Example}

\renewcommand{\thefigure}{\arabic{section}.\arabic{figure}}
\newenvironment{mysubsect}[2][]
{\begin{subsec}\begin{upshape}\begin{bfseries}{#2\vsn.}
\end{bfseries}{#1}}
{\end{upshape}\end{subsec}}
\newcommand{\sect}{\setcounter{figure}{0}\section}

\newenvironment{myeq}[1][]
{\stepcounter{figure}\begin{equation}\tag{\thefigure}{#1}}
{\end{equation}}
\newcommand{\myodiag}[2][]
{\stepcounter{figure}\begin{equation}
     \tag{\thefigure}{#1}\vcenter{\xymatrix@R=10pt@C=1pt{#2}}\end{equation}}
\newcommand{\mypdiag}[2][]
{\stepcounter{figure}\begin{equation}
     \tag{\thefigure}{#1}\vcenter{\xymatrix@R=-1pt@C=5pt{#2}}\end{equation}}
\newcommand{\myqdiag}[2][]
{\stepcounter{figure}\begin{equation}
     \tag{\thefigure}{#1}\vcenter{\xymatrix@R=5pt@C=10pt{#2}}\end{equation}}
\newcommand{\myrdiag}[2][]
{\stepcounter{figure}\begin{equation}
     \tag{\thefigure}{#1}\vcenter{\xymatrix@R=3pt@C=25pt{#2}}\end{equation}}
\newcommand{\mysdiag}[2][]
{\stepcounter{figure}\begin{equation}
     \tag{\thefigure}{#1}\vcenter{\xymatrix@R=0pt@C=20pt{#2}}\end{equation}}
\newcommand{\mytdiag}[2][]
{\stepcounter{figure}\begin{equation}
     \tag{\thefigure}{#1}\vcenter{\xymatrix@R=0pt@C=-27pt{#2}}\end{equation}}
\newcommand{\myudiag}[2][]
{\stepcounter{figure}\begin{equation}
     \tag{\thefigure}{#1}\vcenter{\xymatrix@R=9pt@C=20pt{#2}}\end{equation}}
\newcommand{\myvdiag}[2][]
{\stepcounter{figure}\begin{equation}
     \tag{\thefigure}{#1}\vcenter{\xymatrix@R=16pt@C=36pt{#2}}\end{equation}}
\newcommand{\mywdiag}[2][]
{\stepcounter{figure}\begin{equation}
     \tag{\thefigure}{#1}\vcenter{\xymatrix@R=15pt@C=26pt{#2}}\end{equation}}
\newcommand{\myxdiag}[2][]
{\stepcounter{figure}\begin{equation}
     \tag{\thefigure}{#1}\vcenter{\xymatrix@R=10pt@C=2pt{#2}}\end{equation}}
\newcommand{\myydiag}[2][]
{\stepcounter{figure}\begin{equation}
     \tag{\thefigure}{#1}\vcenter{\xymatrix@R=10pt@C=12pt{#2}}\end{equation}}
\newcommand{\wmono}{ \ar@{>->}[r]}
\newcommand{\wmonovert}{ \ar@{>->}[d]}
\newcommand{\cof}{ \ar@{^{(}->}[r]}

\mathchardef\dashmod="2D

\begin{document}

\title{Purity of quaternionic Conjugation spaces}

\author{Surojit Ghosh}
\address{Department of Mathematics, Indian Institute of Technology, Roorkee, Uttarakhand-247667, India}
\email{surojit.ghosh@ma.iitr.ac.in; surojitghosh89@gmail.com}
\author{Ankit Kumar}
\address{Department of Mathematics, Indian Institute of Technology, Roorkee, Uttarakhand-247667, India}
\email{ankit\_k@ma.iitr.ac.in}

\author{Lakshit Pande}
\address{Department of Mathematics, Indian Institute of Technology, Roorkee, Uttarakhand-247667, India}
\email{lakshit\_p1@ma.iitr.ac.in}

\date{\today}

\begin{abstract}
Conjugation spaces relate the cohomology of a space and its fixed points via a degree-halving isomorphism and admit a characterization in terms of homological purity. We extend this framework to the Klein four group, where the corresponding structures exhibit a degree-quartering behavior governed by Dickson invariants. Under a mild assumption, we prove that quaternionic conjugation spaces are homologically pure. As an application, we show that such spaces are both $\K$-maximal and $\K$-Galois maximal, establishing a connection with Smith--Thom type inequalities in real algebraic geometry.
\end{abstract}

\maketitle

\sect{Introduction}

The study of spaces equipped with group actions often reveals deep and unexpected relationships between the topology of a space and that of its fixed-point locus. A particularly striking instance of this phenomenon arises in the theory of \emph{conjugation spaces}, introduced by Hausmann, Holm, and Puppe~\cite{HHP05}. These spaces are endowed with a $C_2$-action whose cohomology exhibits a remarkable ``halving'' phenomenon: the mod $2$ cohomology of the space and that of its fixed-point set are isomorphic after dividing degrees by two. This structure has strong consequences: it determines the Steenrod algebra action, relates characteristic classes of $X$ and its fixed points, and imposes stringent constraints on the topology of $X$. Further structural properties, including compatibility with Steenrod operations, were studied by Franz and Puppe~\cite{FP10}.

A major conceptual advance was achieved by Pitsch, Ricka, and Scherer~\cite{PRS21}, who reinterpreted conjugation spaces in the framework of equivariant stable homotopy theory. They showed that conjugation spaces are precisely those $C_2$-spaces that are \emph{homologically pure}, meaning that their equivariant suspension spectrum splits into pieces indexed by multiples of the regular representation of $C_2$. This perspective encodes the conjugation structure intrinsically in the equivariant homotopy type and explains the rigidity and functoriality of the theory.

The theory has also benefited from a growing collection of examples and geometric constructions. Classical examples include complex projective spaces, Grassmannians, and flag manifolds~\cite{HHP05}. More recently, Pitsch and Scherer~\cite{PS23} showed that the Floyd manifolds admit conjugation structures, where the cohomology of the fixed-point manifold is obtained from that of the ambient manifold by degree halving, compatibly with Steenrod operations. These developments highlight the ubiquity of conjugation phenomena beyond classical algebraic geometry.

\medskip

In recent work, Kryazhev, Kuznetsov, and Popelensky~\cite{KKP25} extended the idea of conjugation spaces to actions of the Klein four group. In this broader setting, the familiar “degree-halving” phenomenon is replaced by a \emph{degree-quartering} correspondence, together with a section of the restriction map satisfying a quaternionic analogue of the conjugation equation. The spaces that arise in this context are called \emph{quaternionic conjugation spaces}, and they are described using a structure known as a 
$\mathcal{Q}$-frame. As in the classical case, these maps are multiplicative and natural.

\medskip

Quaternionic conjugation spaces arise naturally from $\K$-equivariant analogues of classical constructions. In particular, quaternionic projective spaces, Grassmannians, and flag varieties admit canonical $\K$-actions, making them quaternionic conjugation spaces~\cite{KKP25}. These examples indicate that the theory captures a new layer of symmetry extending the classical conjugation paradigm.

\medskip

The main goal of this paper is to provide a conceptual understanding of quaternionic conjugation spaces via homological purity. We show that every quaternionic conjugation space $X$ is homologically pure, which means that it admits the following decomposition as an $H\uFz$-module:
\[
X\wedge H\uFz \cong \bigvee_iS^{n_i\rho_{\K}}\wedge H\uFz
\]
This demonstrates that the $\mathcal{Q}$-frame structure is intrinsically encoded in the equivariant homotopy type, thereby extending the conceptual framework of~\cite{PRS21}.

\medskip

In a complementary direction, we investigate connections with classical notions of maximality arising in real algebraic geometry. The study of maximal and Galois-maximal spaces originates in the work of Harnack and Krasnov~\cite{Kra83} and \cite{Wil78}, where the topology of real loci is compared to that of complex varieties via Smith--Thom type inequalities. These notions admit natural interpretations in terms of equivariant and Borel cohomology.

We extend these ideas to $\K$-spaces and relate them to quaternionic conjugation structures. We prove that every quaternionic conjugation space is $\K$-maximal, showing that the rigid cohomological structure imposed by an $\mathcal{Q}$-frame forces equality in the corresponding Smith--Thom inequality. This establishes a new connection between equivariant homotopy theory and classical problems in real algebraic geometry.

These results suggest several directions for further research, including a full characterization of quaternionic conjugation spaces via homological purity and extensions to higher elementary abelian $2$-groups.
\medskip

\subsection{Organization} We organize the remainder of the paper as follows.  
In Section~2, we recall the necessary background on Bredon cohomology, along with certain cofibre sequences associated to families of subgroups that will be used throughout the paper.  
In Section~3, we compute the $RO(\K)$-graded Bredon cohomology of various universal spaces.  
Building on these computations, in Section~4 we establish our main result, Theorem~\ref{main1}, concerning quaternionic conjugation spaces and their cohomological purity under suitable assumptions.  
Finally, in Section~5, we relate these $\K$-conjugation spaces to maximal and Galois-maximal spaces arising in real algebraic geometry.

\begin{mysubsect}{Notation}
We fix the following notation throughout the paper.
\begin{enumerate}

\item Square brackets $[\,]$ denote polynomial generators, while angle brackets $\langle\,\rangle$ denote additive generators. For example, $\Fz\langle a, b \rangle[c]$ denotes the algebra with additive generators $a,b$ and polynomial generator $c$; its elements are of the form $l a c^i + m b c^j$ for all $i,j \ge 0$ and $l,m \in \Fz$.

\item The notation $\rho_G$ (respectively, $\bar{\rho}_G$) denotes the regular (respectively, reduced regular) real representation of a finite group $G$.

\end{enumerate}
\end{mysubsect}

\sect{Preliminaries}
\begin{mysubsect}{Equivariant cohomology}
Ordinary cohomology theories take values in abelian groups and are represented by spectra with homotopy concentrated in degree $0$. In the equivariant setting, the analogous role is played by Mackey functors. We briefly recall their definition and relation to equivariant cohomology.

\medskip

Let $\Burn_G$ denote the Burnside category of a finite group $G$, whose objects are finite $G$-sets and whose morphisms are given by the group completion of spans of finite $G$-sets.

\begin{definition}
A \emph{Mackey functor} for $G$ is an additive functor
\[
\uM:\Burn_G^{\op} \longrightarrow \mathrm{Ab}.
\]
\end{definition}

Equivalently, when $G$ is abelian, a Mackey functor $M$ consists of abelian groups $\uM(G/H)$ for each subgroup $H \le G$, together with restriction and transfer maps
\[
\res^H_K: \uM(G/H) \to \uM(G/K),
\qquad
\tr^H_K: \uM(G/K) \to \uM(G/H),
\quad (K \le H),
\]
satisfying the usual functoriality, equivariance, and double coset relations.

\begin{example}
For an abelian group $C$, the constant Mackey functor is given by $C(G/H)=C,\quad \res^H_K=\mathrm{id},\quad \tr^H_K = [H:K].$
\end{example}

\medskip

For a $G$-spectrum $X$, the equivariant homotopy groups assemble into a Mackey functor
\[
\underline{\pi}_n(X)(G/H) := \pi_n(X^H).
\]
Conversely, for any Mackey functor $\uM$, there exists an Eilenberg-Mac Lane $G$-spectrum $H\uM$ with
\[
\underline{\pi}_0(H\uM)=\uM, \qquad \underline{\pi}_n(H\uM)=0 \ (n\neq 0).
\]

This yields an $RO(G)$-graded cohomology theory defined by
\[
H^\alpha_G(X;\uM)
\;\cong\;
[X, \Sigma^\alpha H\uM]^G.
\]
Throughout, we work exclusively with H$\uFz$ coefficients in the context of Bredon (co)homology. 
\medskip

Recall that $RO(G)$ is the Grothendieck group of finite-dimensional $G$-representations. For $\K = \{e,\tau,\tau',\tau\tau'\}$ -- the Klein $4$-group, the real representation ring is
\[
\mathrm{RO}(\K) \cong \mathbb{Z}\{1,\alpha_0,\alpha_1,\beta\},
\]
where $\alpha_0, \alpha_1, \beta \in \mathrm{Hom}_{\Fz}(\K,\Fz)$ are the nontrivial characters with kernels
\[
H_0 = \{e,\tau_1\}, \qquad
H_1 = \{e,\tau_3=\tau_1\tau_2\}, \qquad
K = \{e,\tau_2\}.
\]
\end{mysubsect}
\medskip

\begin{mysubsect}{Some special Bredon cohomology classes}
Let $V$ be a $G$-representation with $V^G=\{0\}$. We denote by $a_V$ the inclusion of the fixed points,  
\[
a_V: S^0 \to S^V.
\]
For a ring spectrum $X$ with a $G$-action, we abuse notation and also denote by $a_V$ its image under the map $S^0 \to X$.

If the representation $V$ contains the trivial representation as a summand, then we set $a_V = 0$.  
Moreover, for any two $G$-representations $V$ and $W$, we have the relation  
\[
a_V a_W = a_{V \oplus W}.
\]

\noindent See \cite[Definition 3.11]{HHR16} for further details.

For an orientable $G$-representation $V: G \to SO(V)$, a choice of orientation induces an isomorphism  
\[
\tilde{H}^{\dim(V)}_G(S^V; \uZ) \cong \Z.
\]
The Thom space of the equivariant bundle $V \to G/G$ is $S^V$. In particular, the restriction map  
\[
\tilde{H}^{\dim(V)}_G(S^V; \uZ) \to \tilde{H}^{\dim(V)}_e(S^{\dim(V)}; \Z)
\]
is an isomorphism.

Utilizing the above isomorphism, for an orientable $G$-representation $V$, we define the orientation class $u_V$ as  
\[
u_V \in \tilde{H}^{V-\dim(V)}_G(S^0 ; \uZ),
\]
the generator that maps to $1$ under the restriction isomorphism. The orientation class satisfies the relations:  
\[
u_{V \oplus 1} = u_V, \quad u_V \cdot u_W = u_{V \oplus W}.
\]
\noindent See \cite[Definition 3.12]{HHR16} for more details.

\begin{remark}
   We use the same notations for the image of the classes $a_V$ and $u_V$ in $\widetilde{H}^{\bigstar}_G(S^0; \uFz)$ under the spectrum map $H \uZ \to H\uFz.$ For this case, we do not need orientable representation for the existence of $u_V.$
\end{remark}
\end{mysubsect}

\begin{mysubsect}{Families and universal spaces} A \emph{family} $\cF$ of subgroups of a finite group $G$ is a nonempty collection of subgroups that is closed under conjugation and passage to subgroups.
\begin{example}
\begin{enumerate}
    \item Let $\mathcal P$ denote the family of all proper subgroups of $G$.
\item 
Let $V$ be a $G$-representation and define
\(
\cF_V = \{H \le G : V^H \neq 0\}.
\)
\item For $B \le G,$ the family $\cF_{\subseteq B}$ is the collection of all subgroups of $B.$ 
\end{enumerate}
\end{example}
To any such family one associates a $G$–CW complex $E\cF$, characterized up to $G$–homotopy equivalence by the fixed-point condition
\[
(E\cF)^H \simeq
\begin{cases}
\ast & \text{if } H \in \cF,\\
\varnothing & \text{otherwise}.
\end{cases}
\]
Equivalently, $E\cF$ is initial among $G$–CW complexes whose isotropy groups lie in $\cF$.

\medskip

\medskip

Associated to any family $\cF$ is the cofibre sequence
\[
E\cF_+ \longrightarrow S^0 \longrightarrow \widetilde{E\cF}.
\]
The cofibre $\widetilde{E\cF}$ is characterized by
\[
(\widetilde{E\cF})^H \simeq
\begin{cases}
\ast & \text{if } H \in \cF,\\
S^0 & \text{otherwise}.
\end{cases}
\]

\begin{definition}  
For a $G$–spectrum $E$, the geometric fixed point $\Phi^G: \Sp^G \to \Sp$ are defined by
\[
\Phi^G(E) := \bigl(\widetilde{E\mathcal P} \wedge E\bigr)^G
\]
is a symmetric monoidal functor, and sends a $G$-spaces $X$ to its fixed point $X^G.$ Thus one observes that $\Phi^G(G/H_+) \simeq \ast,$ if $H < G$ is proper.
\end{definition}

\medskip 
For a character $\alpha \in \{\alpha_0,\alpha_1,\beta\}$ of $\K$, define the associated family
\[
\cF_\alpha := \{ B \le \K \mid \alpha|_B = 1 \}.
\]
Equivalently,
\[
\cF_\alpha = \cF_{\subseteq \ker(\alpha)}.
\]

\begin{lemma}
For each character $\alpha$, there is an equivalence
\[
E\cF_\alpha \simeq S(\infty \alpha).
\]
\end{lemma}

\begin{corollary}
For any $\K$-spectrum $E$, there are equivalences
\[
E^{\wedge}_{\cF_\alpha} \simeq E^{\wedge}_{a_\alpha},
\qquad
E[\cF_\alpha^{-1}] \simeq E[a_\alpha^{-1}].
\]
\end{corollary}

Suppose $\cF \subset \cF'$ are families such that $\cF' \setminus \cF = \{H\}$ for some subgroup $H \le G$. Then there is a cofibre sequence
\[
E\cF_+ \longrightarrow E\cF'_+ \longrightarrow E\cFh,
\]
where $E\cFh$ is characterized by
\begin{myeq}
(E\cFh)^K \simeq
\begin{cases}
S^0 & \text{if } K = H,\\
\ast & \text{otherwise}.
\end{cases}
\end{myeq}

\begin{example}
Let $\mathcal H=\{\{e\}, H_0,H_1\}$ and let $E_{C_2}\mathbb Z/2$ denote the associated universal space. Let $E\K$ denote the universal space for the trivial family. Then the diagram
\[
\xymatrix{
\mathcal{H} \ar[r] & \cP\\
\{e\} \ar[u]\ar[r] & {\cF_{\subseteq K}}\ar[u]
}
\]
induces the diagram of cofibre sequences
\begin{myeq}\label{pullcof}
    \xymatrix{
\cdots\ar[r]&{E_{C_2}\Z/2}_+ \ar[r] &E{\cP}_+ \ar[r] & E\cFk\ar[r]&\cdots\\
\cdots\ar[r]& E{\K}_+ \ar[r]\ar[u]& E{\cF_{\subseteq K}}_+\ar[r]\ar[u]& E\cFk\ar@{=}[u]\ar[r]&\cdots
}
\end{myeq}

which will be used to compute the Bredon cohomology of $E\mathcal P_+$.
\end{example}
\end{mysubsect}
\sect{Bredon cohomology of various universal spaces}
The Bredon cohomology of $\K$-spaces possesses rich algebraic structures owing to the fairly nontrivial morphology of the subgroup lattice of $\K$. In this section, we compute the $\RO(\K)$-graded Bredon cohomology of various universal spaces associated with the subgroup lattice of $\K$.\medskip

\begin{mysubsect}{Cohomology of $E{\cF_{\subseteq K}}_+$} Since one can identify
$E{\cF_{\subseteq K}}\cong E{\K} / K$ with $S(\infty \beta)$, which gives the skeletal filtration of $E{\cF_{\subseteq K}}$. The associated graded filtration is $E{\cF_{\subseteq K}}^{(s)}/E{\cF_{\subseteq K}}^{(s-1)}$ which can be identified with $\underset{e \in I(s)}{\bigvee}{\K}/{K}_+\wedge S^s$.

For $V \in \RO({\K})$, define 
\[E_1^{s,t}(V)=\pi_{t-s}^{{\K}}F(E{\cF_{\subseteq K}}^{(s)}/E{\cF_{\subseteq K}}^{(s-1)}, \Sigma^{-V} H\uFz)\]
By making some identification using adjunctions, one can calculate the $E_2$-page as:
\[
E^{s,V}_2= H^s (B({\K}/K); \pi^{H_1}_{\mathrm{res}_{K}(V)}(H\uFz)) \Rightarrow H^{s - V}_{\K}(E{\cF_{\subseteq K}}; \uFz)
\]
with differentials $d_r: E^{s,V}_r \to E^{s+r,V-r+1}_r$ for $r\geq 2.$ 

Since the action of ${\K}/K$ on the coefficient is trivial (as everything is in $\Fz$). Therefore, the $E_2$-page can be written as 
\[E_2^{\ast, \bigstar}=\Fz[t]\otimes \left[\Fz[a_{\alpha_0},u_{\alpha_0}]\oplus \bigoplus_{j,k \ge 1} \Fz\langle\Sigma^{-1}\frac{1}{a_{\alpha_0}^j u_{\alpha_0}^k}\rangle \right][u_{\beta}^{\pm}, u_{\alpha_1 - {\alpha_0}}^\pm]\]
where $|t|=(1,0),|u_{\alpha_0}|=(0,{{\alpha_0}}-1),|a_{{\alpha_0}}|=(0,{{\alpha_0}})$ and $|u_{\beta}|= (0,{\beta} -1)$ and $|u_{\alpha_1-{\alpha_0}}|= (0,{\alpha_1} -{\alpha_0}).$

One derives that all the differentials are zero and hence the spectral sequence collapses at $E_2$-page. Since the spectral sequence is over $\Fz$, all the extensions are trivial. Thus,
\begin{myeq}\label{comh1}
\widetilde{H}^{\bigstar}_{{\K}}(E{\cF_{\subseteq K}}_+;\uFz) \cong \Fz[t, u_{\beta}^{\pm}, u_{\alpha_1 - {\alpha_0} }^\pm]\otimes \left[\Fz[a_{\alpha_0},u_{\alpha_0}]\oplus \bigoplus_{j,k \ge 1} \Fz\langle{\Sigma^{-1}\frac{1}{a_{\alpha_0}^j u_{\alpha_0}^k}\rangle} \right],
\end{myeq}
\end{mysubsect}
\begin{mysubsect}{Cohomology of $E{\K}_+$}
The computation of $\widetilde{H}^{\bigstar}_{\K}(E{\K}_+;\uFz)$ follows using a homotopy fixed point spectral sequence where the $E_2$-page is given by 
\[E_2^{s, V}=H^s(B\K;\pi_{\dim(V)}(H\Fz)) \Rightarrow \widetilde{H}^{s-V}_{\K}(E{\K}_+;\uFz)\]
with differential $d_2:E_2^{s, V} \to E_2^{s+2, V-1}$.
Since $\pi_{\dim(V)}(H\Fz)$ is concentrated in $\dim(V)=0$. Thus, the spectral sequence collapses at the  $E_2$-page. Therefore,
\begin{myeq}
\widetilde{H}^{\bigstar}_{\K}(E{\K}_+;\uFz) \cong \Fz[t_0, t_1,u_{\alpha_0}^\pm, u_{\alpha_1}^\pm,u_{\beta}^\pm] 
\end{myeq}
where $|t_i|=(1,0)$ and $|u_V|=(0, V-1)$ for $V \in \{\alpha_0, \alpha_1, \beta\}.$
\end{mysubsect}
\begin{mysubsect}{Cohomology of $\widetilde{H}^{\bigstar}_{\K}(E\cFk;\uFz)$} Since to understand the Bredon cohomology of $E\cFk$, it is sufficient to compute the ring homomorphism $q^{\bigstar}$ along with the following cofibre sequence
\begin{myeq}\label{eqcof2}
E{\K}_+\xrightarrow[]{q} E{\cF_{\subseteq K}}_+\to E\cFk
\end{myeq}

\begin{proposition}\label{q1Aloc}
The map $q^\bigstar: \widetilde{H}^\bigstar_{\K} ({E\cF_{\subseteq K}}_{+}; \underline{\mathbb{F}}_2) \to \widetilde{H}^\bigstar_{\K} (E{\K}_{+}; \underline{\mathbb{F}}_2)$ sends:
\[
\begin{aligned}
t &\mapsto t_1, \quad & u_{\alpha_0} &\mapsto u_{\alpha_0}, \\
u_{\alpha_1 - \alpha_0} &\mapsto u_{\alpha_1} u_{\alpha_0}^{-1}, \quad & u_\beta &\mapsto u_\beta, \\
a_{\alpha_0} &\mapsto t_0 u_{\alpha_0}, 
 \quad & \Sigma^{-1}\frac{1}{a_{\alpha_0}^j u_{\alpha_0}^k} &\mapsto 0.
\end{aligned}
\]
\end{proposition}

\begin{proof}
  The natural surjection map 
$
{\K}
\to {\K}/{K}
$,
induces the map $q: E{\K} \to E{\cF_{\subseteq K}}$. In $\Z$-graded Bredon cohomology, it induces
\[
q^\ast:H^\ast_{{\K}} (E{\cF_{\subseteq K}};\uFz)\cong \Fz[t] \to H^\ast_{{\K}}\left(E{\K};\uFz\right) \cong \Fz[t_0, t_1]
 \hspace{1pt},\]
with $q^\ast(t)=t_1$. 
Let $\pi:E{\K}_{+} \to S^{0}$, $\hat{q}: E{\cF_{\subseteq K}}_{+} \to S^0$  be  the collapsing maps. Then we have the following commutative diagram of rings 
\begin{myeq}\label{locsquare}
\xymatrix{\widetilde{H}^\bigstar_{\K}(E{\K}_+;\uFz)   & & \widetilde{H}^\bigstar_{\K}(S^0;\uFz) \ar[dl]^{\hat{q}^\bigstar}\ar[ll]_{\pi^\bigstar} \\ 
& \widetilde{H}^\bigstar_{\K}(E{\cF_{\subseteq K}}_{+};\uFz) \ar[ul]^{q^\bigstar} }
\end{myeq}
Note that by Proposition \ref{pproj1}, $\pi^\bigstar(a_{\alpha_0})=t_0u_{\alpha_0}$, $\pi^\bigstar(a_{\alpha_1})=(t_0-t_1)u_{\alpha_1}$, and $\pi^\bigstar(a_{\beta})=t_1u_{\beta}$.

The action of the ring homomorphism $q^\bigstar$ on $\RO({\K})$-graded classes follows from the fact that $a_{\alpha_0},a_{\alpha_1},a_{\beta},u_{\alpha_0},u_{\alpha_1}$, and 
$u_{\beta}$ map nontrivially under $q^{\bigstar}$. 
\end{proof}

Like the case for $E{\cF_{\subseteq K}}$, one has the above map for $E{\cF_{\subseteq H_i}}$ which gives the new description of the above computed cohomology.

\begin{proposition}\label{pproj1}

$(1)$ The $\RO(\K)$--graded Bredon cohomology of the free $\K$--space $E\K$ is given by
\[
\widetilde{H}^{\bigstar}_{\K}(E{\K}_{+};\uFz)
\cong
\dfrac{
\Fz\!\left[
a_{\alpha_0},a_{\alpha_1},a_{\beta},
u_{\alpha_0}^{\pm1},u_{\alpha_1}^{\pm1},u_{\beta}^{\pm1}
\right]
}{
\bigl(
a_{\alpha_0}u_{\alpha_1}u_{\beta}
+a_{\alpha_1}u_{\alpha_0}u_{\beta}
+a_{\beta}u_{\alpha_0}u_{\alpha_1}
\bigr)
}.
\]

    $(2)$ The $\RO(\K)$--graded Bredon cohomology of the $\K$--space $E{\cF_{\subseteq K}}$ is given by
    \[
 \widetilde{H}^{\bigstar}_{\K}({E{\cF_{\subseteq K}}}_{+};\uFz)\cong\left[\Fz[a_{\alpha_0},u_{\alpha_0}]\oplus \bigoplus_{j,k \ge 1} \Fz\langle\Sigma^{-1}\frac{1}{a_{\alpha_0}^j u_{\alpha_0}^k}\rangle\right] [a_{\beta},u^{\pm}_{\beta},u^{\pm}_{\alpha_1-\alpha_0}].
 \]
 
 $(3)$ The $\RO(\K)$--graded Bredon cohomology of the $\K$--space $E\cFk$ is given by
 
 \begin{align*}
 \widetilde{H}^{\bigstar}_{\K}(E\cFk;\uFz)\cong &\Fz\left\langle{\Sigma^{-1}\frac{1}{a^{k_1}_{\alpha_0}u^{k_2}_{\alpha_0}}\colon k_1,k_2\geq 1}\right\rangle[a_{\beta},u^{\pm}_{\beta},u^{\pm}_{\alpha_{1}-\alpha_0}]\\  
&\bigoplus \Sigma^{-1} u_{\alpha_0}^{-1}\,\Fz\!\left[
a_{\alpha_0},\, a_{\alpha_1},\, a_\beta,\,
\Bigl(\tfrac{u_{\alpha_1}}{u_{\alpha_0}}\Bigr)^{\pm},\,
u_{\beta}^{\pm},\, u_{\alpha_0}^{-1}
\right].
\end{align*}
\end{proposition}
\end{mysubsect}

\begin{mysubsect}{Homotopy groups of geometric fixed point}
Consider the isotropy--separation cofibre sequence
\[
E{\K}_{+}\xrightarrow{\ \pi\ } S^0 \longrightarrow \widetilde{E\K}.
\]
This induces a ring homomorphism in $\RO(\K)$--graded Bredon cohomology
\[
\pi^\bigstar \colon \widetilde{H}^{\bigstar}_{\K}(S^0;\uFz)
\longrightarrow
\widetilde{H}^{\bigstar}_{\K}(E{\K}_{+};\uFz).
\]

\begin{proposition}\label{prop:image of s0 in ek4}
The image of $\pi^\bigstar$ is the subring of $ \widetilde{H}^{\bigstar}_{\K}(E{\K}_{+};\uFz)$ generated by
\[
a_{\alpha_0},a_{\alpha_1},a_{\beta},\quad
u_{\alpha_0},u_{\alpha_1},u_{\beta},\quad
\frac{u_{\alpha_0}u_{\alpha_1}}{u_\beta},\quad
\frac{u_{\beta}u_{\alpha_1}}{u_{\alpha_0}},\quad
\frac{u_{\beta}u_{\alpha_0}}{u_{\alpha_1}}.
\]
\end{proposition}

\begin{proof}
In $\widetilde{H}^{\bigstar}_{\K}(E{\K}_{+};\uFz)$ the classes $u_V$ are invertible for all nontrivial representations $V\in\{\alpha_0,\alpha_1,\beta\}$. Hence the map $\pi^\bigstar$ factors through the localization of $\widetilde{H}^{\bigstar}_{\K}(S^0;\uFz)$ obtained by inverting these classes.

The image is therefore generated by those elements of the localized ring that arise from $\widetilde{H}^{\bigstar}_{\K}(S^0;\uFz)$ under this map. Concretely, this yields the stated generators together with the single relation coming from the Euler class relation among $\alpha_0, \alpha_1,\beta$. This completes the proof.
\end{proof}

\begin{lemma}\label{lgeomfixed}
The $\RO(\K)$-graded homotopy of $\widetilde{E\cP}\wedge H\uFz$ is 
\[
\pi_{\bigstar}(\widetilde{E\cP}\wedge H\uFz)\cong \dfrac{
\Fz\!\left[
u_{\alpha_0},u_{\alpha_1},u_{\beta},
a_{\alpha_0}^{\pm1},a_{\alpha_1}^{\pm1},a_{\beta}^{\pm1}
\right]
}{
\bigl(
a_{\alpha_0}u_{\alpha_1}u_{\beta}
+a_{\alpha_1}u_{\alpha_0}u_{\beta}
+a_{\beta}u_{\alpha_0}u_{\alpha_1}
\bigr)
}.
\]

In particular, the geometric fixed-point homotopy ring of \(H\uFz\) is given by
\[
\pi_\ast\bigl(\Phi^{\K}(H\uFz)\bigr)
\cong
\dfrac{\Fz[x_{\alpha_0},x_{\alpha_1},x_{\beta}]}
{\bigl(x_{\alpha_0}x_{\alpha_1}+x_{\alpha_0}x_{\beta}+x_{\alpha_1}x_{\beta}\bigr)}.
\]
Under this identification, the degree 1 class \(x_V\) is represented by \(u_V/a_V\) for each
\(V\in\{\alpha_0,\alpha_1,\beta\}\).
\end{lemma}

\begin{proof}
    Since $\widetilde{E\cP}\wedge H\uFz \simeq H\uFz[a_{\alpha_0}^{\pm1},a_{\alpha_1}^{\pm1},a_{\beta}^{\pm1}]$, thus it is enough to compute $\pi_\ast\bigl(\Phi^{\K}(H\uFz)\bigr)$. Hence the result follows from \cite{HK17}. 
\end{proof}
\begin{remark}
There is an $\Fz$-vector space decomposition
\[
\pi_\ast\bigl(\Phi^{\K}(H\uFz)\bigr)
\cong
\Fz[x_{\alpha_0},x_\beta]
\oplus
x_{\alpha_1}\Fz[x_{\alpha_1},x_\beta].
\]
Equivalently, the set of monomials
\[
\{x_{\alpha_0}^i x_\beta^j \mid i,j\ge 0\}
\cup
\{x_{\alpha_1}^m x_\beta^n \mid m\ge 1,\ n\ge 0\}
\]
forms an $\Fz$-basis of
\[
\frac{\Fz[x_{\alpha_0},x_{\alpha_1},x_\beta]}
{\bigl(x_{\alpha_0}x_{\alpha_1}+x_{\alpha_0}x_\beta+x_{\alpha_1}x_\beta\bigr)}.
\]
\end{remark}
\end{mysubsect}
\begin{mysubsect}{Bredon cohomology of $E\cP$}
Let us consider $\cF=\{e, H_0, H_1\}$ and $\cP=\{e, H_0, H_1, K\}$, then we have the cofibre sequences:
\begin{myeq}\label{eqcof1}
{E_{C_2}\Z/2}_+ \to E{\cP}_+ \to E\cFk
\end{myeq}
and 
\begin{myeq}\label{eqcof2}
E{\K}_+\to E{\cF_{\subseteq K}}_+\to E\cFk
\end{myeq}
where $E_{C_2}\Z/2:=E\cF$.

The diagram \eqref{pullcof}, induces the following diagram of long exact sequences in $RO(\K)$-Bredon cohomology:
\begin{myeq}\label{interdiag}
\xymatrix{
\cdots 
\widetilde{H}^{\bigstar}_{\K}(E\cFk) \ar[r] \ar@{=}[d] &
\widetilde{H}^{\bigstar}_{\K}(E\cP_+) \ar[r] \ar[d] &
\widetilde{H}^{\bigstar}_{\K}({E_{C_2}\Z/2}_+) \ar[r]^{\delta^u} \ar[d]^{q^{\bigstar}_{\mathrm{tot}}} &
\widetilde{H}^{\bigstar+1}_{\K}(E\cFk) \ar@{=}[d] 
\cdots
\\
\cdots 
\widetilde{H}^{\bigstar}_{\K}(E\cFk) \ar[r] &
\widetilde{H}^{\bigstar}_{\K}(E{\cF_{\subseteq K}}_+) \ar[r]_{q^\bigstar} &
\widetilde{H}^{\bigstar}_{\K}(E{\K}_+) \ar[r]^{\delta} &
\widetilde{H}^{\bigstar+1}_{\K}(E\cFk) \ 
\cdots
}
\end{myeq}

\begin{proposition}\label{main}
 There is an equivalence of rings
   \[
   \widetilde{H}^{\bigstar}_{\K}({E_{C_2}\Z/2}_+;\uFz) \cong \frac{\left[\Fz[a_{\beta},u_{\beta}]\oplus \bigoplus_{j,k \geq 1} \Fz\left\langle\Sigma^{-1}\frac{1}{ a^j_{\beta}u^k_{\beta}}\right\rangle\right] [a_{\alpha_0}, u_{\alpha_0}, a_{\alpha_1} , u_{\alpha_1}, v^\pm ]}{(vu_\beta -u_{\alpha_0}u_{\alpha_1}, va_\beta -(a_{\alpha_0}u_{\alpha_1} +u_{\alpha_0}a_{\alpha_1}))}
   \]
   where $v$ is the invertible class in degree $(-1+{\alpha_0}+{\alpha_1}-\beta).$
\end{proposition}

\begin{proof}
    See \cite{GK26}.
\end{proof}

\begin{proposition}
    The image of $q^\bigstar_{\mathrm{tot}} : \widetilde{H}^{\bigstar}_{\K}({E_{C_2}\Z/2}_+;\uFz) \to \widetilde{H}^{\bigstar}_{\K}(E{\K}_+;\uFz)$ in \eqref{interdiag} is given by
    \[
\dfrac{
\Fz\;[
a_{\alpha_0},a_{\alpha_1},a_{\beta},
u_{\alpha_0},u_{\alpha_1},u_{\beta}, (\frac{u_{\alpha_0}u_{\alpha_1}}{u_{\beta}})^{\pm}
]
}{
(
a_{\alpha_0}u_{\alpha_1}u_{\beta}
+a_{\alpha_1}u_{\alpha_0}u_{\beta}
+a_{\beta}u_{\alpha_0}u_{\alpha_1}
)}.
    \]
\end{proposition}

\begin{proof}
 Since \(q^{\bigstar}_{\mathrm{tot}}\) is the composite of the maps \(q^{\bigstar}_{i}\) and \(\pi^{\bigstar}_{i}\), whose descriptions are given explicitly in~\cite{GK26}, the claim follows from the commutativity of the diagram
\begin{myeq}\label{locsquare1}
\xymatrix{
& \widetilde{H}^\bigstar_{\K}(E{\cF_{\subseteq H_1}}_{+};\uFz) \ar[dl]_{q_1^\bigstar} \\
\widetilde{H}^\bigstar_{\K}(E{\K}_+;\uFz) & &
\widetilde{H}^\bigstar_{\K}({E_{C_2}\Z/2}_+;\uFz)
\ar[dl]^{\pi_0^\bigstar}\ar[ll]_{q^\bigstar_{\mathrm{tot}}} \ar[ul]_{\pi_1^\bigstar} \\
& \widetilde{H}^\bigstar_{\K}(E{\cF_{\subseteq H_0}}_{+};\uFz) \ar[ul]^{q_0^\bigstar}
}
\end{myeq}
which completes the proof.
\end{proof}


\begin{lemma}\label{lem:cohomology of ep}\begin{enumerate}
    \item The kernel of the boundary map $\delta^{u}: \widetilde{H}^{\bigstar}_{\K}({E_{C_2}\Z/2}_+;\uFz) \to 
\widetilde{H}^{\bigstar+1}_{\K}(E\cFk;\uFz)$ in \eqref{interdiag} is given by
 \[
    \frac{\left[\Fz[a_{\beta},u_{\beta}]\oplus \bigoplus_{j,k \geq 1; \ell \in \Z} \Fz\left\langle{v^{\ell}\Sigma^{-1}\frac{1}{ a^j_{\beta}u^k_{\beta}}}\right\rangle\right] [a_{\alpha_0}, u_{\alpha_0}, a_{\alpha_1} , u_{\alpha_1}, v ,\frac{u^2_{\alpha_0}}{v},\frac{u^2_{\alpha_1}}{v}]}{(vu_\beta -u_{\alpha_0}u_{\alpha_1}, va_\beta -(a_{\alpha_0}u_{\alpha_1} +u_{\alpha_0}a_{\alpha_1}))}.
   \]
   \item The cokernel of the boundary map $\delta^{u}$ is given by
   \begin{align*}
   \Fz\!\left\langle\, 
\Sigma^{-1}\,\frac{1}{a_{\alpha_0}^{\,j} u_{\alpha_0}^{\,k}} 
\;:j,k\geq 1\right\rangle
\bigl[a_{\beta},\, u_{\beta}^{\pm},\, u_{\alpha_1 - {\alpha_0}}^{\pm}\bigr]
\oplus\Sigma^{-1}u^{-1}_{\alpha_0}\Fz[a_{\beta},a_{\alpha_0},a_{\alpha_1},(\frac{u_{\alpha_0}}{u_{\alpha_1}})^{\pm},u^{-1}_{\beta},u^{-1}_{\alpha_0}]\\
\oplus\Fz\left\langle{\Sigma^{-1}\frac{u_{\beta}}{u^2_{\alpha_0}},\Sigma^{-1}\frac{u_{\beta}u_{\alpha_1}}{u^2_{\alpha_0}}}\right\rangle[a_{\beta},a_{\alpha_0},a_{\alpha_1},\frac{u_{\alpha_1}}{u_{\alpha_0}},\frac{u_{\beta}}{u_{\alpha_0}},u^{-1}_{\alpha_0}]\\
\oplus\Fz\left\langle{\Sigma^{-1}\frac{u_{\beta}}{u^2_{\alpha_1}},\Sigma^{-1}\frac{u_{\beta}u_{\alpha_0}}{u^2_{\alpha_1}}}\right\rangle[a_{\beta},a_{\alpha_0},a_{\alpha_1},\frac{u_{\alpha_0}}{u_{\alpha_1}},\frac{u_{\beta}}{u_{\alpha_1}},u^{-1}_{\alpha_1}]\\
\oplus\Fz\left\langle{\Sigma^{-1}\frac{u_{\beta}}{u^2_{\alpha_0}u_{\alpha_1}},\Sigma^{-1}\frac{u_{\beta}}{u^2_{\alpha_1}u_{\alpha_0}}}\right\rangle[a_{\beta},a_{\alpha_0},a_{\alpha_1},\frac{u_{\beta}}{u_{\alpha_0}},\frac{u_{\beta}}{u_{\alpha_1}},u^{-1}_{\alpha_0},u^{-1}_{\alpha_1}].
   \end{align*}
\end{enumerate}

\end{lemma}

\begin{proof}
    $(1)$ Note that for an class $x \in \widetilde{H}^{\bigstar}_{\K}({E_{C_2}\Z/2}_+)$, if there exists another class $y \in \widetilde{H}^{\bigstar}_{\K}(E{\cF_{\subseteq K}}_+)$ such that $q^\bigstar_{\mathrm{tot}}(x)=q^\bigstar(y)$, then $\delta^u{x}=0.$ This follows from the diagram chase. Thus one observe that the classes $a_{\alpha_0},a_{\alpha_1},a_{\beta},
u_{\alpha_0},u_{\alpha_1},u_{\beta},v$ have nonzero class in $\widetilde{H}^{\bigstar}_{\K}(E{\cF_{\subseteq K}}_+)$ such that above condition holds, whereas classes multiple of $\Sigma^{-1}\frac{1}{ a^j_{\beta}u^k_{\beta}}$ maps to zero through $q^\bigstar_{\mathrm{tot}}$, hence they are in the kernel of $\delta^u.$

$(2)$ For the computation of the cokernel, note that class $v^{-1}$ map non-trivially to $\Sigma^{-1}\frac{u_{\beta}}{u_{\alpha_0}u_{\alpha_1}}$, then the above expression follows from the $\widetilde{H}^{\bigstar}_{\K}(S^0)$--module structure of the boundary map $\delta^{u}$. 
\end{proof}


\begin{proposition}\label{prop:surjective map s0 to ep}
Let $\,\gamma \in \RO(\K)$ satisfy $\dim\gamma^{\K}\leq0$. Then the map $q^\gamma \colon \widetilde{H}_{\K}^\gamma(S^0)
\longrightarrow \widetilde{H}_{\K}^\gamma(E\cP_+)$ induced by the collapse map $q \colon E\cP_+ \to S^0$ is surjective.
\end{proposition}

\begin{proof}
    We have the cofibre sequence
    \begin{myeq}
        \cdots \longrightarrow \Sigma^{-1}\widetilde{E{\cP}} \longrightarrow E{\cP}_{+}\overset{q}{\longrightarrow} S^0 \longrightarrow \widetilde{E{\cP}}\longrightarrow \cdots
    \end{myeq}
    which induces the following long exact sequence in cohomology.
    \begin{myeq}\label{elxep}
       \cdots \longrightarrow  \hstar[\gamma]{\EPtilde} \longrightarrow \hstar[\gamma]{S^0}\overset{q^\gamma}{\longrightarrow} \hstar[\gamma]{\EPplus}\overset{\partial^\gamma}{\longrightarrow}  \hstar[\gamma+1]{\EPtilde}\longrightarrow \cdots 
    \end{myeq}

    \noindent Since $\EPtilde \simeq S^0[a_{\overline{\rho}_{\K}}^{-1}]$, we have $\hstar[\gamma]{\EPtilde} \cong \widetilde{H}_{\K}^{\dim \gamma^{\K}}(\EPtilde)$. The above long exact sequence at $\gamma=0$ yields 
    \[
    \cdots \hstar[0]{\EPtilde} \longrightarrow \hstar[0]{S^0}\cong\mathbb{F}_2 \overset{q^0}{\longrightarrow} \hstar[0]{\EPplus}\cong  \mathbb{F}_2\overset{\partial^{0}}{\longrightarrow}\hstar[1]{\EPtilde}\longrightarrow \hstar[1]{S^0}\cong 0.
    \]
    The map $q^0$ is an isomorphism, which forces $\hstar[1]{\EPtilde} \cong  0$ and moreover, one verifies $\hstar[n]{\EPtilde} \cong 0$ for all $n\le 0$. Therefore, the long exact sequence \eqref{elxep} at any $\RO(\K)$-degree $\gamma$ satisfying $\dim \gamma^{\K}\leq0$ reads as 
    \[\cdots \longrightarrow  \hstar[\gamma]{S^0}\overset{q^\gamma}{\longrightarrow}  \hstar[\gamma]{\EPplus}\longrightarrow \hstar[\gamma+1]{\EPtilde} \cong 0,
    \]
    whence we infer that the map $q^\gamma$ is a surjection.
\end{proof}
\end{mysubsect}

\sect{Homological purity}

The purpose of this section is to place quaternionic conjugation spaces within the framework of equivariant stable homotopy theory and to relate their defining properties to homological purity. We begin by recalling the definition of quaternionic conjugation spaces and then study their behavior via isotropy separation.

\begin{mysubsect}{Quaternionic conjugation spaces}
Let $X$ be a $\K$-space. We work throughout with cohomology with coefficients in $\uFz$.

\begin{definition}[\text{\cite[Section 2.1]{KKP25}}]
A \emph{quaternionic conjugation space} (or \emph{$\K$-conjugation space}) is a $\K$-space $X$ equipped with the following data:
\begin{enumerate}
    \item An additive isomorphism
    \[
    \kappa \colon H^{4*}(X) \xrightarrow{\;\cong\;} H^*(X^{\K}),
    \]
    which divides degrees by a factor of $4$.
    
    \item An additive section
    \[
    \sigma \colon H^*(X) \longrightarrow H^*_{\K}(X)
    \]
    of the restriction map
    \[
    \rho \colon H^*_{\K}(X) \longrightarrow H^*(X),
    \]
    satisfying the \emph{quaternionic conjugation equation}: for every $x \in H^{4n}(X)$, we have
    \begin{myeq}\label{eq:Q-conjugation}
        r\circ\sigma\bigl(x\bigr)
        =
        \kappa(x)\, p_0^{\,n}
        + \sum_{0< i+j\le n} \kappa_{i, j}(x)p_1^{j}p_0^{n-j-i},
    \end{myeq}
    where
    \[
    r \colon H^*_{\K}(X) \longrightarrow H^*(X^{\K}) \otimes H^*(B\K)
    \]
    is the restriction to fixed points,
    \[
    p_0 = tt'(t+t'), \quad p_1=t^2+tt'+{t'}^{2}\, \in\, H^*(B\K;\Fz)\cong \Fz[t,t']
    \]
    are the degree-$3$, and degree-$2$ Dickson invariants, respectively, and 
    \[
    \kappa_{i,j}(x)\in H^{n+3i+j}(X^{\K}).
    \]
\end{enumerate}
\end{definition}
\medskip

\noindent Such a collection of data is called a $\mathcal{Q}$-frame in \cite{KKP25}. It relates the ordinary cohomology of $X$, its fixed-point space $X^{\K}$, and its Borel equivariant cohomology. From a stable equivariant viewpoint, this structure is expected to reflect a strong decomposition property, analogous to the homological purity condition in the $C_2$-case.
\begin{remark}
As in the classical case of conjugation spaces, the maps $\kappa$ and $\sigma$ are in fact multiplicative, and satisfy naturality with respect to $\K$-equivariant maps. Furthermore, Kryazhev, Kuznetsov, and Popelensky (\cite{KKP25}) show that the values of the coefficients $\kappa_{i,j}(x)\in H^{n+3i+j}(X^{\K})$ in the conjugation equation \eqref{eq:Q-conjugation} are completely determined by the value of the degree quartering map $\kappa(x)$. Precisely, we have 
    \[
    \kappa_{i,j}(x) = Sq(j,i)(\kappa(x))
    \]
    where $Sq(j,i)$ are elements of the Milnor basis for the Steenrod algebra $\mathcal{A}_2$. This observation is in harmony with the classical case (for a $C_2$-conjugation space), where the expanded conjugation equation reads as
    \[
     r\bigl(\sigma(x)\bigr)
        =
        \kappa(x)t^{\,n}
        + \sum_{i=1}^n (Sq^i\kappa(x))t^{n-i}
    ,\]
    as established by Hausmann, Holm, and Puppe (\cite{HHP05}).
\end{remark}
\end{mysubsect}
\medskip

\begin{mysubsect}{The isotropy separation diagrams}
Let \(X\) be a \({\K}\)-space. The inclusion map of families $\{e\}\xhookrightarrow{}\mathcal P$ induces a map between cofibre sequences of the corresponding universal spaces.
\[
\xymatrix{
\Sigma^{-1}\EPtilde\ar[r]  &   \EPplus\ar[r] & S^0\ar[r]  &  \EPtilde \\
\Sigma^{-1}\widetilde{E\K}\ar[r]\ar[u]  &   E{\K}_+\ar[r]\ar[u] & S^0\ar[r]\ar@{=}[u]  &  \widetilde{E\K}\ar[u] \\
}
\]

\noindent Smashing the above diagram with the inclusion of fixed points $X^{\K}\hookrightarrow X$, and applying \(\RO({\K})\)-graded Bredon cohomology in \(\uFz\) coefficients, we obtain the following comparison diagram

\begin{myeq}\xymatrix@C=-1pt@R=15pt{
  &  & \hstar{X^{\K}}\ar[dl]\ar[dr]&&\\
  & \hstar{E{\K}_+\wedge X^{\K}}\ar[dl] & & \hstar{\EPplus\wedge X^{\K}}\ar[ll]\ar[dr]&\\
  \hstar[\bigstar+1]{\widetilde{E\K}\wedge X^{\K}} &  &  \hstar{X}\ar'[u][uu]\ar[dl]\ar[dr] &  &  \hstar[\bigstar+1]{\EPtilde\wedge X^{\K}}\\
  & \hstar{E{\K}_+\wedge X}\ar[dl]\ar[uu] & & \hstar{\EPplus\wedge X}\ar[ll]\ar[dr]\ar[uu]&\\
  \hstar[\bigstar+1]{\widetilde{E\K}\wedge X}\ar[uu] &  &  &  &  \hstar[\bigstar+1]{\EPtilde\wedge X}\ar[uu]
  }
\end{myeq}

\medskip

\noindent where the rows along the slanted arrows are exact. Using that (\cite[Lemma 13]{Car99}) for a space $B$ with trivial $\K$-action one has
\[
\hstar{A\wedge B} \cong \hstar{A}\otimes \widetilde{H}^*(B),
\]
 together with the equivalence $\EPtilde\wedge X\simeq \EPtilde\wedge X^{\K}$, the diagram simplifies to

\begin{myeq}\label{eq:main}
    \xymatrix@C=-17pt@R+=2em{
  &  & \hstar{S^0}\otimes \widetilde{H}^*(X^{\K})\ar[dl]_{j_1^\bigstar}\ar[dr]^{j_2^\bigstar}&&\\
  & \hstar{E{\K}_+}\otimes \widetilde{H}^*(X^{\K})\ar[dl] & \phantom{pp}& \hstar{\EPplus}\otimes \widetilde{H}^*(X^{\K})\ar[ll]^>(0.8){{g}^\bigstar}\ar[dr]&\\
  \hstar[\bigstar+1]{\widetilde{E\K}\wedge X^{\K}} &  &  \hstar{X}\ar'[u][uu]_{\kappa^\bigstar}\ar[dl]\ar[dr] &  &  \hstar[\bigstar+1]{\EPtilde\wedge X^{\K}}\\
  & \hstar{\EKplus\wedge X}\ar[dl]\ar[uu]^{r_1} & & \hstar{\EPplus\wedge X}\ar[ll]\ar[dr]\ar[uu]^{r_2}&\\
  \hstar[\bigstar+1]{\EKtilde\wedge X}\ar[uu] &  &  &  &  \hstar[\bigstar+1]{\EPtilde\wedge X}\ar[uu]^\cong,
  }
\end{myeq}

\noindent at the heart of which is the map
\[
\hstar{X}\longrightarrow\hstar{\EPplus\wedge X}\longrightarrow\hstar{\EKplus\wedge X}.
\]
 Passing to the underlying nonequivariant cohomology for this map, we obtain:

\[
\xymatrix@R=15pt{
    & \hstar{X}\ar[dl]\ar[dr]\ar'[d][dd]^<<{\rho_X}\\
    \hstar{\EKplus\wedge X}\ar[dd]^{\rho}    &&    \hstar{\EPplus\wedge X}\ar[dd]\ar[ll] \\
    & \widetilde{H}^{|\bigstar|}(X)[u_V^\pm] \ar[dl]_\cong\ar[dr]^\cong\\
    \widetilde{H}^{|\bigstar|}(X)[u_V^\pm] && \widetilde{H}^{|\bigstar|}(X)[u_V^\pm]\ar[ll]^\cong
}
\]
\medskip\\
where $\widetilde{H}^{|\bigstar|}(X)[u_V^\pm]$ denotes the localization $\widetilde{H}^{|\bigstar|}(X)\otimes\Fz[u_{\alpha_0}^\pm,u_{\alpha_1}^\pm,u_{\beta}^\pm]$. Observe that since $\EKplus\wedge X$ is a free \K-space, we have 
\[
\hstar{\EKplus\wedge X}\cong \widetilde{H}^{|\bigstar|}(\EKplus\wedge_{\K}X)[u_V^\pm].
\]

\noindent If $X$ is equipped with a $\mathcal{Q}$-frame, then we have a $\Fz[u_{\alpha_0}^\pm,u_{\alpha_1}^\pm,u_{\beta}^\pm]$--module map 
\begin{myeq}\label{eq:sigma}
\sigma:\widetilde{H}^{|\bigstar|}(X)[u_V^\pm]\longrightarrow \widetilde{H}^{|\bigstar|}(\EKplus\wedge_{\K}X)[u_V^\pm],
\end{myeq}

\noindent which is a section for the restriction map $\rho$, and which when composed with the map $r_1$ in diagram \eqref{eq:main} reproduces the conjugation equation within the $\RO(\K)$-graded module $\hstar{\EKplus}\otimes\widetilde{H}^*(X^{\K})$.

\begin{lemma}\label{llifting}
Let $X$ be a $\K$-conjugation space, equipped with a map $\sigma'$, which lifts the section $\sigma$ in equation \eqref{eq:sigma} along the map $\hstar{\EPplus\wedge X}\longrightarrow\hstar{\EKplus\wedge X}$. Then for every class
$x\in \widetilde{H}^{4n}(X),$ there exists a class
\[
\widetilde{x}\in \hstar[n\rho_{\K}]{X;\uFz}
\]
such that
\[
\rho_X(\widetilde{x})
=
(u_{\alpha_0}u_{\alpha_1}u_{\beta})^{n}x
\in
\widetilde{H}^*(X)\bigl[u_{\alpha_0}^{\pm1},u_{\alpha_1}^{\pm1},u_{\beta}^{\pm1}\bigr].
\]
\end{lemma}

\begin{proof}
Let \(x\in H^{4n}(X)\). Since \(X\) is a $\K$-conjugation space, we have the conjugation equation
\begin{align*}
    r_1\circ\sigma\big((u_{\alpha_0}u_{\alpha_1}u_{\beta})^{n}x\big) &= (u_{\alpha_0}u_{\alpha_1}u_{\beta})^{n}\:r_1\circ\sigma(x)\\
    &= \kappa(x)\,(u_{\alpha_0}u_{\alpha_1}u_{\beta})^{n} p_0^{\,n}
        + \sum_{0< i+j\le n} \kappa_{i, j}(x)(u_{\alpha_0}u_{\alpha_1}u_{\beta})^{n}p_1^{j}p_0^{n-j-i}
\end{align*}
where\[p_0 = t_0t_1(t_0+t_1), \qquad p_1=t_0^2+t_1^2+t_1t_2= (t_0+t_1)t_0 + (t_0+t_1)t_1+t_0t_1.
\]
Using the convention
$
t_0=a_{\alpha_0}u_{\alpha_0}^{-1},
\quad
t_1=a_{\beta}u_{\beta}^{-1},
\quad
t_0+t_1=a_{\alpha_1}u_{\alpha_1}^{-1},
\,$
we obtain
\[
(u_{\alpha_0}u_{\alpha_1}u_{\beta})p_0
=a_{\alpha_0}a_{\alpha_1}a_{\beta}, \quad\text{and}\quad 
 (u_{\alpha_0}u_{\alpha_1}u_{\beta})p_1
=a_{\alpha_0}a_{\alpha_1}u_{\beta}+a_{\alpha_1}a_{\beta}u_{\alpha_0}+a_{\alpha_0}a_{\beta}u_{\beta}.
\]
Observe that the above terms are monomials in $\Fz[a_{\alpha_0},a_{\alpha_1},a_{\beta},u_{\alpha_0},u_{\alpha_1},u_{\beta}]$.
Therefore, we have
\begin{align*}
    (u_{\alpha_0}u_{\alpha_1}u_{\beta})^{n}\:r_1\circ\sigma(x) \:&=\: \kappa(x)\big(u_{\alpha_0}u_{\alpha_1}u_{\beta}\, p_0\big)^{n}
        +\\& \sum_{0< i+j\le n} \kappa_{i, j}(x)\big(u_{\alpha_0}u_{\alpha_1}u_{\beta}\,p_1\big)^{j}\big(u_{\alpha_0}u_{\alpha_1}u_{\beta}\,p_0\big)^{n-j-i}\big(u_{\alpha_0}u_{\alpha_1}u_{\beta}\big)^{n-i}
\end{align*}
which clearly belongs to the image of $j_1^\bigstar$ in $\hstar{\EKplus}\otimes\widetilde{H}^*(X^{\K})$ by Proposition \ref{prop:image of s0 in ek4}.
Furthermore, by assumption, the term $r_1\circ\,\sigma\big((u_{\alpha_0}u_{\alpha_1}u_{\beta})^{n}x\big)$ has a preimage $r_2\circ\sigma'\big((u_{\alpha_0}u_{\alpha_1}u_{\beta})^{n}x\big)$ in $\hstar{\EPplus}\otimes\widetilde{H}^*(X^{\K})$, which, by construction, is an element of the submodule 
\[
\bigoplus_{i=0}^{3n} \hstar[n\overline{\rho}_{\K}-\:i]{\EPplus}\otimes\widetilde{H}^{n+i}(X^{\K}) \: \subseteq\: \hstar{\EPplus}\otimes\widetilde{H}^*(X^{\K}).
\] 
 By Proposition \ref{prop:surjective map s0 to ep}, the map $\hstar[n\overline{\rho}_{\K}-i]{S^0}\longrightarrow \hstar[n\overline{\rho}_{\K}-i]{\EPplus}$ is surjective for all $i\geq 0$. Therefore, the element $r_2\circ\sigma'\big((u_{\alpha_0}u_{\alpha_1}u_{\beta})^{n}x\big)$ lies in the image of $j_2^\bigstar$, which (from diagram \eqref{eq:main}) implies that $\sigma'\big((u_{\alpha_0}u_{\alpha_1}u_{\beta})^{n}x\big)$ has a preimage $\tilde{x}$ in $\hstar[n\rho_{\K}]{X}$. Evidently, this satisfies
\[
\rho_X(\tilde{x}) = \rho\circ\sigma\big((u_{\alpha_0}u_{\alpha_1}u_{\beta})^{n}x\big)= (u_{\alpha_0}u_{\alpha_1}u_{\beta})^{n}x .
\]
Hence the result follows.
\end{proof}


\begin{lemma}\label{lrest}
Let $X$ be a $\K$-conjugation space, then for any order $2$ subgroup $L \le \K$, the restriction $\mathrm{res}_{L}(X)$ is a $C_2$-conjugation space.
\end{lemma}

\begin{proof}
    It suffices to treat $L=\langle \tau_1\rangle$, since the other cases are
identical. By the discussion in \cite[\S4.1]{KKP25}, if
$X$ is $\mathcal{Q}$-framed, then $X$, regarded only as a $L$-space, carries
an $H$-frame $(\sigma_1,\kappa_1)$.

Let
\[
p \colon E{\K}\times_{\K} X \longrightarrow E\langle\tau_1\rangle\times_{\tau_1} X
\]
be the natural map of Borel constructions. Then
\[
p^*\circ \sigma_{\K}=\sigma_1
\]
by \cite[Proposition~4.3]{KKP25}, and
\[
\kappa_{\K}=\kappa_{23}\circ \kappa_1
\]
by \cite[Proposition~4.4]{KKP25}, where $\kappa_{23}$ is the $H$-frame isomorphism on
$X^{\tau_1}$ with respect to the residual involution
\[
\tau_{23}=\tau_2|_{X^{\tau_1}}=\tau_3|_{X^{\tau_1}}.
\]

Since $\kappa_{\K}$ and $\kappa_{23}$ are isomorphisms, $\kappa_1$ is an
isomorphism as well. Hence $(\sigma_1,\kappa_1)$ is an $H$-frame on $X$
for the involution $\tau_1$, so $\mathrm{res}^{\K}_L X$ is a $C_2$-conjugation space.

The same argument applies to any order-two subgroup $L\le {\K}$.
\end{proof}

\begin{theorem}\label{main1}
    Let $X$ be a finite type $\K$-conjugation space, satisfying the assumptions of Lemma \ref{llifting}. Then $X$ is homologically pure, i.e. as a $H\uFz$-module, there exists a decomposition
    \[
    X\wedge H\uFz \simeq \bigvee_{i} S^{n_i\rho_{{\K}}}\wedge H\uFz.
    \]
\end{theorem}

\begin{proof}
By Lemma \ref{llifting}, for any class \(x \in \widetilde{H}^{4\ast}(X;\Fz)\) there exists a lift \(\tilde{x} \in \widetilde{H}^{\ast \rho_{\K}}(X;\uFz)\). Since \(\widetilde{H}^{\ast}(X;\Fz)\) is finitely generated -- say by a set \(\{x_i\}_{i\in F}\) -- we obtain a map
\[
f \colon X \longrightarrow \bigvee_{i\in F} S^{n_i\rho_{\K}} \wedge H\uFz
\]
in \(\mathrm{Sp}^{\K}\), where the restriction of \(f\) to each summand \(S^{n_i\rho_{\K}} \wedge H\uFz\) represents the class \(\tilde{x}_i\). This induces an \(H\uFz\)-module map
\[
\hat{f} \colon X \wedge H\uFz \longrightarrow \bigvee_{i\in F} S^{n_i\rho_{\K}} \wedge H\uFz.
\]

To show that \(\hat{f}\) is an equivalence of \(H\uFz\)-modules, it suffices to verify that the induced map \(\Phi^{L}(\hat{f})\) is an equivalence in the corresponding geometric fixed-point category. By Lemma \ref{lrest}, the restriction \(\mathrm{res}^{\K}_{L}(X)\) is a \(C_2\)-conjugation space for any order two subgroup $L$. Consequently, it is enough to check that \(\Phi^{\K}(\hat{f})\) is an equivalence. Note that the geometric fixed point functor is monoidal and on spaces it only computes the fixed point, we derive that 
\begin{myeq}\label{eqmod}
    \Phi^{\K}(\hat{f}): X^{\K}\wedge H\Fz \wedge_{H\Fz}\Phi^{\K}(H\uFz)\to \bigvee_{i\in F} S^{n_i} \wedge H\Fz \wedge_{H\Fz}\Phi^{\K}(H\uFz).
\end{myeq}

By Lemma \ref{lgeomfixed}, we have
\[
\pi_\ast\bigl(\Phi^{\K}(H\uFz)\bigr)
\cong
\frac{\Fz[x_{\alpha_0},x_{\alpha_1},x_{\beta}]}
{\bigl(x_{\alpha_0}x_{\alpha_1}+x_{\alpha_0}x_{\beta}+x_{\alpha_1}x_{\beta}\bigr)}.
\]
Under this identification, the degree \(1\) class \(x_V\) is represented by \(u_V/a_V\) for each
\(V \in \{\alpha_0,\alpha_1,\beta\}\). The maximal ideals of this ring are
\[
(x_{\alpha_0},x_{\alpha_1},x_{\beta}), \quad
(x_{\alpha_0}-1,x_{\alpha_1},x_{\beta}), \quad
(x_{\alpha_0},x_{\alpha_1}-1,x_{\beta}), \quad
(x_{\alpha_0},x_{\alpha_1},x_{\beta}-1).
\]
Let us denote the ring \(\pi_\ast\bigl(\Phi^{\K}(H\uFz)\bigr)\) by \(R\).

On homotopy groups, the map \(\Phi^{\K}(\hat{f})\) in \eqref{eqmod} induces
\[
\Phi^{\K}(\hat{f})_\ast \colon 
\widetilde{H}_\ast(X^{\K}; \Fz)\otimes_{\Fz} R 
\longrightarrow 
\bigoplus_{i\in F} \widetilde{H}_{\ast}(S^{n_i})\otimes_{\Fz} R.
\]
For any maximal ideal \(\mathfrak{m} \subset R\), consider the induced map
\begin{myeq}\label{inducedmap}
    \Phi^{\K}(\hat{f})_\ast \otimes_R R/\mathfrak{m} \colon 
\widetilde{H}_\ast(X^{\K}; \Fz) 
\longrightarrow 
\bigoplus_{i\in F} \widetilde{H}_{\ast}(S^{n_i}).
\end{myeq}

To prove that the map
\(\Phi^{\K}(\hat{f})_\ast \otimes_R R/\mathfrak{m}\) is an isomorphism,
we first observe that the canonical map of spectra
\(H\uFz \to \widetilde{E\cP}\wedge H\uFz\) induces, after applying \(\Phi^{\K}\),
a map of ring spectra
\[
H\Fz \longrightarrow \Phi^{\K}(H\uFz),
\]
since \(\Phi^{\K}\) is symmetric monoidal. In particular,
\(\Phi^{\K}(H\uFz)\) is naturally an \(H\Fz\)-module. It follows that
\(\Phi^{\K}(H\uFz)\) splits as a wedge of suspensions of \(H\Fz\),
namely
\begin{myeq}\label{eqsplitting}
\Phi^{\K}(H\uFz)
\;\simeq\;
(\bigvee_{i,j\ge 0}\Sigma^{i+j}H\Fz)
\;\vee\;
(\bigvee_{\substack{m\ge 1\\ k\ge 0}}\Sigma^{m+k}H\Fz),
\end{myeq}
corresponding to the monomial basis
\(
\{x_{\alpha_0}^i x_\beta^j \mid i,j\ge 0\}
\cup
\{x_{\alpha_1}^m x_\beta^k \mid m\ge 1,\ k\ge 0\}
\)
of \(R=\pi_\ast(\Phi^{\K}(H\uFz))\).

Now, for a class \(\widetilde{x_i}\in H^{n_i\rho_{\K}}_{\K}(X;\uFz)\), 
its geometric fixed point 
\[
\Phi^{\K}(\widetilde{x_i})
\in [X^{\K}_+\wedge S^{-n_i},\, \Phi^{\K}(H\uFz)]
\]
identifies, via \eqref{eqsplitting}, with an element of
\(
H^{n_i}(X^{\K};\Fz)\otimes_{\Fz} R,
\) which also identifies with the image of \(\kappa^\bigstar(\widetilde{x_i})\) under
\(\mathrm{id}\otimes \gamma\), where
\(
\gamma \colon \pi_{\ast}^{\K}(H\uFz)\longrightarrow \pi_{\ast}(\Phi^{\K}(H\uFz))=R.
\)

Thus, the image of $\Phi^{\K}(\widetilde x_i)$ in degree $n_i$ is exactly \(
\kappa(x_i)\in \widetilde H^{n_i}(X^{\K};\Fz).\) Therefore, the dual map of \eqref{inducedmap}
\begin{myeq}\label{eqdual}
    \mathrm{Hom}_{\Fz}(\Phi^{\K}(\hat{f})_\ast \otimes_R R/\mathfrak{m} , \Fz) \colon 
\bigoplus_{i\in F} \widetilde{H}^{\ast}(S^{n_i})  
\longrightarrow \widetilde{H}^\ast(X^{\K}; \Fz)
 \end{myeq}
sends the generator of
\(
\widetilde H^{n_i}(S^{n_i};\Fz)
\)
to $\kappa(x_i)$. Since $\kappa\colon \widetilde H^{4*}(X;\Fz)\to \widetilde H^*(X^{\K};\Fz)$ is an isomorphism and $\{x_i\}_{i\in F}$ is a basis, the set
\(
\{\kappa(x_i)\}_{i\in F}
\)
is a basis of $\widetilde H^*(X^{\K};\Fz)$. It follows that the map in \eqref{eqdual} is an isomorphism, and hence $\Phi^{\K}(\hat{f})_\ast \otimes_R R/\mathfrak{m}$ itself is an isomorphism. Thus by an application of Nakayama's lemma (see \cite[Proposition 2.8]{AM69} ) it follows that \(\Phi^{\K}(\hat{f})_\ast\) is itself an isomorphism.
\end{proof}

\end{mysubsect}

\sect{Maximal and Galois-maximal spaces}

The notions of maximal and Galois-maximal spaces arise naturally in real algebraic geometry, where a real algebraic variety $V$ is studied through its complex locus $V(\mathbb{C})$ endowed with the involution given by complex conjugation. The idea of maximality can be traced back to Harnack, while Galois-maximal spaces were later introduced by Krasnov as a refinement that encodes how the $C_2$-action governs the cohomology of the real locus $V(\mathbb{R})$. These concepts play a fundamental role in classification problems, revealing deep connections between topology and algebraic geometry.

From a topological viewpoint, maximality is often formulated in terms of equivariant cohomology. In particular, via the Borel construction, maximality corresponds to a freeness condition over $H^*(BC_2;\Fz)$. Finer structure is detected by $RO(C_2)$-graded equivariant cohomology, whose computational accessibility has significantly improved in recent years. In this section, we investigate these notions for $\K$-spaces.

\medskip

For a $\K$-space $X$, the classical Smith--Thom inequality asserts that
\[
\sum_{i=0}^{\dim X^{\K}} \dim_{\Fz} H^i (X^{\K}; \Fz)
\;\le\;
\sum_{i=0}^{\dim X} \dim_{\Fz} H^i (X; \Fz).
\]
We say that $X$ is $\K$-\emph{maximal} if equality holds.

\medskip

The $\K$-action further endows $H^*(X;\Fz)$ with the structure of an $\Fz[\K]$-module, leading to the refined inequality
\[
\sum_{i=0}^{\dim X^{\K}} \dim_{\Fz} H^i (X^{\K}; \Fz)
\;\le\;
\frac{1}{2}\sum_{i=0}^{\dim X} \dim_{\Fz} H^1(\K; H^i(X; \Fz)).
\]
We say that $X$ is \emph{$\K$-Galois-maximal} if equality holds. The factor $\tfrac{1}{2}$ arises from the fact that for any trivial $\K$-module $M$, one has
\(
\dim_{\Fz} H^1(\K; M) = 2\,\dim_{\Fz} M.
\)

\medskip

Our goal in this section is to clarify the relationship between conjugation spaces and these maximality conditions. In particular, we establish the following result.

\begin{proposition}
Let $X$ be a $\K$-conjugation space. Then $X$ is $\K$-maximal.
\end{proposition}

\begin{proof}
By definition of a ${\K}$-conjugation space, there is an additive isomorphism
\begin{myeq}\label{eqkappa}
\kappa\colon H^{4n}(X;\Fz)\xrightarrow{\cong} H^n (X^{\K};\Fz).
\end{myeq}

for each $n\ge 0$.  Taking dimensions over $\Fz$, we obtain $\dim_{\Fz} H^n(X^{\K};\Fz)=\dim_{\Fz} H^{4n}(X;\Fz).$

Now, in the standard definition of a ${\K}$-conjugation space, one assumes that $H^i(X;\Fz)=0
\text{ whenever } i\not\equiv 0 \pmod 4.$
Therefore,
\[
\sum_{i\ge 0}\dim_{\Fz} H^i(X;\Fz)
=
\sum_{n\ge 0}\dim_{\Fz} H^{4n}(X;\Fz).
\]
Thus using \eqref{eqkappa}, we obtain
\[
\sum_{i\ge 0}\dim_{\Fz} H^i(X;\Fz)
=
\sum_{n\ge 0}\dim_{\Fz} H^n(X^{\K};\Fz)
=
\sum_{i\ge 0}\dim_{\Fz} H^i(X^{\K};\Fz).
\]
Hence $X$ is ${\K}$-maximal.
\end{proof}

\begin{proposition}
Let $X$ be a ${\K}$-conjugation space. Then $X$ is Galois-maximal. 
\end{proposition}

\begin{proof}

Since the $\K$-action on the fixed-point space $X^{\K}$ is trivial, the induced action on $H^*(X^{\K};\Fz)$ is likewise trivial. Via the isomorphism $\kappa$, this implies that the $\K$-action on $H^{4*}(X;\Fz)$ is trivial.

Let $M$ be a finite-dimensional $\Fz$-vector space equipped with the trivial $\K$-action. A standard computation in group cohomology shows that
\[
\dim_{\Fz} H^1({\K};M)=2\,\dim_{\Fz} M.
\]

Applying this with $M = H^{4n}(X;\Fz)$ and using \eqref{eqkappa}, we obtain
\[
\dim_{\Fz} H^1\!\bigl(\K; H^{4n}(X;\Fz)\bigr)
=
2\, \dim_{\Fz} H^{4n}(X;\Fz).
\]

On the other hand, by the graded isomorphism \eqref{eqkappa}
\[
\dim_{\Fz} H^n(X^{\K};\Fz)=\dim_{\Fz} H^{4n}(X;\Fz).
\]

Combining these identities, we obtain
\[
\dim_{\Fz} H^n(X^{\K};\Fz)
=
\frac{1}{2}\dim_{\Fz} H^1\!\bigl({\K}; H^{4n}(X;\Fz)\bigr).
\]

Summing over all $n$ and using that $H^i(X;\Fz)=0$ unless $i\equiv 0\pmod{4}$, we may reindex $i=4n$ to obtain
\[
\sum_{i\ge 0}\dim_{\Fz} H^i(X^{\K};\Fz)
=
\frac{1}{2}\sum_{i\ge 0}\dim_{\Fz} H^1\!\bigl({\K}; H^i(X;\Fz)\bigr),
\]
which proves that $X$ is Galois-maximal.
\end{proof}

\end{document}